\documentclass[12pt,reqno]{amsart}

\usepackage{amsmath,amsthm,amssymb,comment,fullpage}

\usepackage{times}

\usepackage[T1]{fontenc}

\usepackage{mathrsfs}

\usepackage{latexsym}

\usepackage[dvips]{graphics}

\usepackage{epsfig}

\usepackage{amsmath,amsfonts,amsthm,amssymb,amscd}

\input amssym.def

\input amssym.tex

\usepackage{color}

\usepackage{hyperref}

\usepackage{url}

\usepackage{breakurl}

\usepackage{tikz}
\usepackage{verbatim}

\newcommand{\bburl}[1]{\textcolor{blue}{\url{#1}}}

\numberwithin{equation}{section}

\newtheorem{thm}{Theorem}[section]

\theoremstyle{plain}

\newtheorem{definition}[thm]{Definition}

\newtheorem{lemma}[thm]{Lemma}

\newtheorem{theorem}[thm]{Theorem}

\newtheorem*{theorem*}{Theorem}

\newtheorem{remark}[thm]{Remark}

\newcommand\be{\begin{equation}}

\newcommand\ee{\end{equation}}

\newcommand\bea{\begin{eqnarray}}

\newcommand\eea{\end{eqnarray}}

\newcommand\bi{\begin{itemize}}

\newcommand\ei{\end{itemize}}

\newcommand\ben{\begin{enumerate}}

\newcommand\een{\end{enumerate}}

\newcommand\bc{\begin{center}}

\newcommand\ec{\end{center}}

\newcommand\ba{\begin{array}}

\newcommand\ea{\end{array}}

\newcommand\frakfamily{\usefont{U}{yfrak}{m}{n}}

\DeclareTextFontCommand{\textfrak}{\frakfamily}

\newcommand{\hr}[1]{\href{#1}{\url{#1}}}

\newcounter{row}
\newcounter{col}
\newcommand\setrow[5]{
  \setcounter{col}{1}
  \foreach \n in {#1, #2, #3, #4, #5} {
    \edef\x{\value{col} - 0.5}
    \edef\y{7.5 - \value{row}}
    \node[anchor=center] at (\x, \y) {\n};
    \stepcounter{col}
  }
  \stepcounter{row}
}

\title{Invariance of the Sprague-Grundy Function for Variants of Wythoff's Game}

\author{Madeleine Weinstein}
\email{\textcolor{blue}{\href{mailto:mweinstein@g.hmc.edu}{mweinstein@hmc.edu}}}
\address{Department of Mathematics, Harvey Mudd College, Claremont, CA 91711 }

\keywords{Combinatorial game theory, Nim, Wythoff, Sprague-Grundy function}

\date{\today}

\begin{document}

\maketitle

\begin{abstract}
We prove three conjectures of Fraenkel and Ho regarding two classes of variants of Wythoff's game. The two classes of variants of Wythoff's game feature restrictions of the diagonal moves. Each conjecture states that the Sprague-Grundy function is invariant up to a certain nim-value for a subset of that class of variant of Wythoff's game. For one class of variants of Wythoff's game, we prove that the invariance of the Sprague-Grundy function extends beyond what was conjectured by Fraenkel and Ho. 
\end{abstract}

\section{Introduction}

In this paper, we prove invariance properties of the Sprague-Grundy function for variants of Wythoff's game. We first state the rule sets of the variants of Wythoff's game. Next, we review background on the Sprague-Grundy function. We also state the invariance properties Fraenkel and Ho \cite{FH} found that lead to the conjectures of further invariance.

\subsection{Rule Sets of Games}
The game of $2$-pile Nim is an impartial game in which two players alternately remove any number of tokens from either of two piles. The game ends when both piles are empty, and the last player able to make a move wins. 
We can conceptualize $2$-pile Nim as being played on a grid of positions marked by coordinates $(a,b)$ where $a$ and $b$ are nonnegative. From a position $(a,b)$, one may move vertically to a position $(a,b-s)$ with $s>0$ or horizontally to a position $(a-s,b)$. In generalizations of Nim, we call such horizontal and vertical moves "Nim moves."
The game of Wythoff allows an additional diagonal move. That is, from a position $(a,b)$, in addition to making a vertical or horizontal Nim move we may move to a position $(a-s,b-s)$. Many variants of Wythoff's game have been studied, with rule sets that either restrict the legal moves of Wythoff's game or allow additional moves. Fraenkel and Ho \cite{FH} looked for games in which the losing positions are translations of the losing positions of Wythoff's game. In the study of this question, they introduced the three classes $ \{W_k\}$, $\{W_{k,l}\}$  and $\{T_k\}$ of variants of Wythoff's game. 
In the class $\{W_k\}_{k \ge 0}$, all Nim moves are allowed but the diagonal move is restricted as follows. A diagonal move from $(a,b)$ to $(a-s,b-s)$ is allowed so long as $\min(a-s,b-s) \ge k$. 
In the class $\{W_{k,l}\}_{0 \leq k \leq l}$, a diagonal move from $(a,b)$ to $(a-s,b-s)$ is allowed so long as $\min(a-s,b-s) \ge k$ and $\max(a-s,b-s) \ge l$. Note that $W_{l,l}$ has the same rule set as $W_l$. 
Lastly, the class $\{T_k\}_{k \ge 0}$  restricts the diagonal moves allowed in $W_1$. Let $a \leq b$. A move from $(a,b)$ to $(a-s,b-s)$ with $a-s>0$ is allowed so long as 
\[ \left|  \left\lfloor \frac{b-s}{a-s} \right\rfloor - \left\lfloor \frac{b}{a}  \right\rfloor   \right| \leq k. \]
We note that $T_{\infty}$ has the same rule set as $W_1$. 

\subsection{Sprague-Grundy Function}

\begin{definition}
The \textbf{nim-value (Sprague-Grundy value)} of a position is defined inductively as follows: The nim-value of all terminal positions (positions from which no move may be made) is 0. The nim-value of any other position $(a,b)$ is the minimum excluded natural number of the set of nim-values of positions reachable in one move from $(a,b)$, that is, the smallest number in the set $\{0,1,2,\dots \}$ that is not the nim-value of some position reachable in one move from $(a,b)$. The \textbf{Sprague-Grundy function} for a game gives the nim-value of a given position. A \textbf{g-position} is a position with nim-value $g$. 
\end{definition}

If a position has nonzero nim-value and the player who makes the next move (the move starting at this position) can win, then the position is called an \textit{N-position}. If a position has a nim-value of $0$, with optimal play only the player who played just previously to that move can win, and the position is called a \textit{P-position}.

Knowledge of the Sprague-Grundy function of an individual combinatorial game extends further than just allowing for determination of a winning strategy for that game: The Sprague-Grundy function of the sum of combinatorial games can be quickly computed from the Sprague-Grundy functions of each of the components.

\subsection{Previous Work}
Wythoff \cite{Wy} found the $P$-positions of Wythoff's game, which involve taking the floor function of a quantity involving the golden ratio. As there is symmetry across the line $y=x$, we list only positions $(a,b)$ with $a\leq b$. 

\begin{definition}
Let $\phi=\frac{1+\sqrt{5}}{2}$, the golden ratio, $A_n=\lfloor \phi n \rfloor$, and $B_n=\lfloor \phi^2 n \rfloor$. 
\end{definition}

\begin{theorem}
(Wythoff  \cite{Wy}) The $P$-positions $(a,b)$ with $a\leq b$ of Wythoff's game form the set $\{ (A_n,B_n) | n \ge 0  \}$. 
\end{theorem}
Fraenkel and Ho \cite{FH} found the $P$-positions of $\{W_k\}$, $\{W_{k,l}\}$, and $\{T_k\}$. In fact, the motivation for introducing these games was to answer the question of when translations of $P$-positions of Wythoff's game are $P$-positions. 

\begin{theorem}
(Fraenkel and Ho) For each $k \ge 0$, the $P$-positions $(a,b)$ with $a \leq b$ of $W_k$  form the set 
\[  \{(i,i) | 0 \leq i <k \} \cup \{(A_n+k, B_n+k)  | n \ge 0 \}. \] 
\end{theorem}

\begin{theorem}\label{ppositionswkl}
(Fraenkel and Ho) Let $k$ and $l$ be nonnegative integers with $k \leq l$. The $P$-positions $(a,b)$ with $ a\leq b$ of $W_{k,l}$ form the set 
\[  \{(i,i) | 0 \leq i <l \} \cup \{(A_n+l, B_n+l)  | n \ge 0 \}. \]
\end{theorem}

Note that the $P$-positions of $W_{k,l}$ are independent of $k$, and equal the $P$-positions of $W_l$. 

\begin{theorem}
(Fraenkel and Ho) For each $k \ge 0$, the $P$-positions $(a,b)$ with $a \leq b$ of the game $T_k$ form the set 
\[ \{(0,0)\} \cup \{ (A_n+1, B_n+1) | n \ge 0 \}. \]
\end{theorem}

Note that in both $W_{k,l}$ and $T_k$, the $P$-positions are independent of $k$. Fraenkel and Ho conjectured further invariance of the Sprague-Grundy functions of games within the class $\{W_{k,l}\}$ for different $k$ and of those within the class $\{T_k\}$. Note that the $P$-positions of $T_k$ equal those of $W_1$, and the rule set of $T_k$ restricts the diagonal moves allowed in $W_1$. This leads to conjectures about further invariance of the Sprague-Grundy function between the games $T_k$ and $W_1$. In each case, invariance holds for $g$-positions up to a certain bound depending on the parameters of the game. We state and prove such conjectures, as well as invariance properties of the Sprague-Grundy function of $W_{k,l}$ beyond what was conjectured by Fraenkel and Ho.

\section{The Class $W_{k,l}$}

Unlike in the abovementioned variants of Wythoff's game, in the game of Nim it is easy to compute the nim-value of a position without recursion. The operation that finds the nim-value of a given position is called the \textit{nim sum}. Our proof of the invariance property of the Sprague-Grundy function of the games $\{W_{k,l}\}$ relies upon considering regions in which the nim-values of a position is just the nim sum of its coordinates. 

\begin{definition}
The \textit{nim sum} $x \oplus y$ of a position $(x,y)$ is the binary digital sum of $x$ and $y$, that is, the sum when both numbers are written in binary and then added without carrying. Equivalently, it is the "exclusive or" or XOR of $x$ and $y$. 
\end{definition}

\begin{definition}
The forbidden region of a given game of the form $W_k$ or $W_{k,l}$ is the part of the grid that cannot be entered on a diagonal move. 
\end{definition}

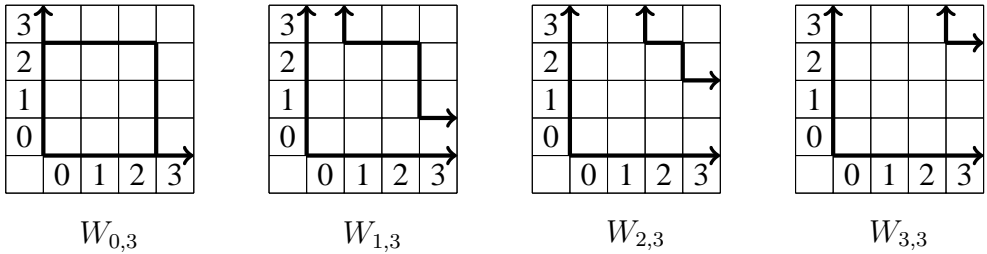
\begin{figure}[h!]
\caption{Forbidden Regions of $W_{k,3}$ for $k=0,1,2,3$}
\begin{tikzpicture}[scale=.5]

  \begin{scope}
    \draw (0, 0) grid (5, 5);
    \draw[ultra thin] (0, 0) grid (5, 5);
	\draw[ultra thick,->] (1,1) -- (5,1);
	\draw[ultra thick,->] (1,1) -- (1,5);
	\draw[ultra thick](1,4) --(4,4);
	\draw[ultra thick](4,1) --(4,4);
    \setcounter{row}{3}
    \setrow {3}{}{ }{}{ }
    \setrow {2}{}{ }{}{ }
    \setrow {1}{}{ }{}{ }
    \setrow {0}{}{ }{}{ }
    \setrow { }{0}{1}{2}{3}

    \node[anchor=center] at (8, -2) {};
    \node[anchor=center,below=0.1cm] at (2.75,-0.25){$W_{0,3}$};
  \end{scope}

    \begin{scope}[xshift=7cm]
    \draw (0, 0) grid (5, 5);
    \draw[ultra thin] (0, 0) grid (5, 5);
	\draw[ultra thick,->] (1,1) -- (5,1);
	\draw[ultra thick,->] (1,1) -- (1,5);
    \draw[ultra thick,->] (2,4) -- (2,5);
    \draw[ultra thick,->] (4,2) -- (5,2);
    
	\draw[ultra thick](2,4) --(4,4);
	\draw[ultra thick](4,2) --(4,4);
    \setcounter{row}{3}
    \setrow {3}{}{ }{}{ }
    \setrow {2}{}{ }{}{ }
    \setrow {1}{}{ }{}{ }
    \setrow {0}{}{ }{}{ }
    \setrow { }{0}{1}{2}{3}
    \node[anchor=center] at (8, -2) {};
    \node[anchor=center,below=0.1cm] at (2.75,-0.25)
   {$W_{1,3}$};

  \end{scope}

    \begin{scope}[xshift=14cm]
    \draw (0, 0) grid (5, 5);
    \draw[ultra thin] (0, 0) grid (5, 5);
	\draw[ultra thick,->] (1,1) -- (5,1);
	\draw[ultra thick,->] (1,1) -- (1,5);
    \draw[ultra thick,->] (3,4) -- (3,5);
    \draw[ultra thick,->] (4,3) -- (5,3);
    
	\draw[ultra thick](3,4) --(4,4);
	\draw[ultra thick](4,3) --(4,4);
    \setcounter{row}{3}
    \setrow {3}{}{ }{}{ }
    \setrow {2}{}{ }{}{ }
    \setrow {1}{}{ }{}{ }
    \setrow {0}{}{ }{}{ }
    \setrow { }{0}{1}{2}{3}
    
    \node[anchor=center] at (8,-2)
 {};
    \node[anchor=center] at (2.75, -1.12) {$W_{2,3}$};
  \end{scope}
  
    \begin{scope}[xshift=21cm]
    \draw (0, 0) grid (5, 5);
    \draw[ultra thin] (0, 0) grid (5, 5);
	\draw[ultra thick,->] (1,1) -- (5,1);
	\draw[ultra thick,->] (1,1) -- (1,5);
    \draw[ultra thick,->] (4,4) -- (4,5);
    \draw[ultra thick,->] (4,4) -- (5,4);
    
    \setcounter{row}{3}
    \setrow {3}{}{ }{}{ }
    \setrow {2}{}{ }{}{ }
    \setrow {1}{}{ }{}{ }
    \setrow {0}{}{ }{}{ }
    \setrow { }{0}{1}{2}{3}
    \node[anchor=center] at (8,-2)
 {};
    \node[anchor=center] at (2.75, -1.12) {$W_{3,3}$};
  \end{scope}

\end{tikzpicture}
\end{figure}

The proof of Theorem \ref{fraenkelho1} and its extensions all rely upon showing that the forbidden regions of two games share a region which contains a $g$-position in every row and column for $g$ up to a certain bound. The presence of the $g$-positions in this region preempts the presence of any $g$-positions in regions that are part of the forbidden region of one game but not the other, thus rendering the differences in the rule sets of the pair of games irrelevant with respect to the location of $g$-positions.

\begin{theorem}\label{fraenkelho1}
(Conjecture 1 of Fraenkel and Ho) Let $k<k^{'} \leq l$. For every integer $g$ in the range $0 \leq g \leq l-k^{'}$, the two games $W_{k,l}$ and $W_{k^{'},l}$ have the same sets of positions with nim-value $g$. 
\end{theorem}

\begin{proof}
The respective forbidden regions of the two games $W_{k,l}$ and $W_{k^{'},l}$ each contain the region 
$[0,k^{'}-1] \times [0,l-1]$
because for any $(x,y)$ in this region, we have $\max(x,y) \leq l$. 
In any rectangle $[0,a] \times [0,b]$, with $0 \leq a \leq b$ of the grid of nim-values of $2$-pile Nim, in every column, there will be a $g-$position for $g \in [0,b-a]$. We can see this as follows. The nim-value of a position $(x,y)$ in 2-pile Nim is given by $x \oplus y$. The definition of $\oplus$ as the XOR operation gives us the two properties that $x \oplus x=0$ and that $x \oplus y \leq x+y$. Suppose $x \oplus y=g$. Taking the nim sum with $x$ on each side gives $y=x \oplus g \leq x+g$. So for $x \in [0,a]$ and $g \in [0,b-a]$, we have $y \leq b$. In particular, for $0 \leq g<l-k^{'}$, the region $[0,k^{'}-1] \times [0,l-1]$ has a $g$-position in every column. 

Consider the rest of the forbidden regions for $W_{k,l}$ and $W_{k^{'},l}$. The part of the forbidden region with $y \ge l$ consists entirely of columns $x$ for $x \in [0,k^{'}-1]$. In no such column can there be any $g$-positions for $0 \leq g \leq l-k{'}$ with $y$-coordinate $y \ge l$ because there is a $g-$position with $y$-coordinate $y \leq l-1$ which is reachable by a vertical Nim move from any position above it in the column $x$. 
Symmetrically, we may argue that the respective forbidden regions of two games $W_{k,l}$ and $W_{k^{'},l}$ each contain the region $[0,l-1] \times [0, k^{'}-1]$ and thus the part of the forbidden regions with $x \ge l$ contains no $g$-positions for $0 \leq g \leq l-k{'}$. 
So for $0 \leq g \leq l-k{'}$, $g$-positions are only located in regions in which positions are either accessible by a diagonal move in both $W_{k,l}$ and $W_{k^{'},l}$ or in neither $W_{k,l}$ nor $W_{k^{'},l}$. Thus the difference in rules between $W_{k,l}$ and $W_{k^{'},l}$ never creates a difference in access to these $g$-positions, so $W_{k,l}$ and $W_{k^{'},l}$ have the same $g$-positions for  $0 \leq g \leq l-k{'}$. 
\end{proof}

\begin{remark}
Note that for general $k< k^{'} \leq l$, the bound for $g$ given in Theorem \ref{fraenkelho1} is tight. For instance, the games $W_{0,2}$ and $W_{1,2}$ have different $2$-positions. For $y=1$, the $2$-position in $W_{0,2}$ is located at $(4,1)$ and the $2$-position in $W_{1,2}$ is located at $(3,1)$. 
\end{remark}

Next, we state an observation about the regularity of the location of $g$-positions less than a given power of two in $2$-pile Nim that allows us to extend Fraenkel and Ho's \cite{FH} conjecture and prove further invariance of the Sprague-Grundy function for the game $W_{k,l}$ in certain circumstances.

\begin{lemma}\label{everyrow}
For $g< 2^j$, the grid $[0,2^j-1] \times [0,2^j-1]$ for the nim-values of 2-pile Nim has a $g$-position in every row and column. 
\end{lemma}

\begin{proof}
By symmetry across the line $y=x$, we need only to prove the statement for each row. 
Consider the nim-values in the $i^{th}$ row. They are obtained by taking the nim sum of each of the integers $k \in [0,2^j-1]$ with $i$. The nim sum is binary addition without carrying, so adding two numbers less than a given power of two will produce a number less than said power of two; as $0 \leq i,k \leq 2^j-1$, we have $0 \leq i \oplus k  \leq 2^j-1$. Furthermore, the operation of "adding" $i$ to $k$ with the nim sum is an involution. Thus the operation of adding $i$ to $k \in [0,2^j-1]$ simply permutes the set $\{0,1, \dots, 2^j-1\}$. So for $g<2^j$ we have a $g$-position in the $i^{th}$ row, and the lemma is proved. 
\end{proof}

\begin{theorem}\label{powersoftwo}
Let $2^m \leq l$.  For $0 \leq g<2^m$, $0 \leq k \leq l$, the $g$-positions of $W_{k,l}$ equal those of $W_l$.  
\end{theorem}

\begin{proof}
Let $g<2^m$. 
We show that the set of $g$-positions in the forbidden regions for each game $W_{k,l}$ is identical to the set of $g$-positions in the forbidden region of $W_l$ for all $g<2^m$. 

Consider the grid $[0,2^m-1] \times [0,2^m-1]$ in the grid of nim-values of $W_{k,l}$. For positions in this grid, no diagonal moves can be made because $\max(i,j)<2^m \leq l$ for all $(i,j)$ in this grid. Thus this grid is the grid of Nim. By Lemma \ref{everyrow}, for $g<2^m$, every row and column of the grid has a $g-$position. 
Consider the rest of the forbidden region for $W_{k,l}$ outside the grid $[0,2^m-1] \times [0,2^m-1]$. For any $(i,j)$ in this region with $j>2^m-1$, we have that $i<l$. The part of the column of $[0,2^m-1] \times [0,2^m-1]$ below $(i,j)$ contains a $g-$position for all $g<2^m$, so $(i,j)$ cannot be a $g-$position for $g<2^m$. Similarly, for any $(i,j)$ in this region with $i>2^m-1$, we have that $j<l$, and in the row of $[0,2^m-1] \times [0,2^m-1]$ to the left of $(i,j)$ there is a $g$-position for all $g<2^m$, so $(i,j)$ cannot be a $g-$position for $g<2^m$. 
Thus there are no $g$-positions for $g<2^m$ in the rest of the forbidden region. 

Therefore, for $0 \leq g<2^m$, $g$-positions are only located in regions in which positions are either accessible by a diagonal move in both $W_{k,l}$ and $W_{l,l}$ or in neither $W_{k,l}$ nor $W_{l,l}$. Thus the difference in rules between $W_{k,l}$ and $W_{l,l}$ never creates a difference in access to these $g$-positions, so $W_{k,l}$ and $W_{l,l}$ have the same $g$-positions for  $0 \leq g<2^m$.

Note that $W_{l,l}$ has an identical set of rules to $W_l$. Thus the set of $g$-positions in each game $W_{k,l}$ is identical to the set of $g$-positions in $W_l$ for all $g<2^m$.

\end{proof}

We have proven the invariance property of the $g$-positions for the games $W_{k,l}$ without actually finding a formula for these positions. In general, it appears to be hard to find an explicit formula for $g$-positions of $W_{k,l}$ with $g \ge 1$. Fraenkel and Ho \cite{FH} provide a recursive formula for the $1$-positions of $W_k$, where the $1$-positions of $W_{k+2}$ are obtained from those of $W_k$. Fraenkel and Ho give an explicit formula for the $1$-positions of $W_1$. Blass and Fraenkel \cite{BF} give a recursive algorithm for computing the $1$-positions of $W_0$ (Wythoff's game), but there does not appear to be an explicit formula in the literature. So while Fraenkel and Ho provide an explicit formula for the $1$-positions of $W_k$ with $k$ odd, no such formula appears to exists for $k$ even. 
Computer explorations indicate that for $l$ even, the set of $1$-positions of $W_{k,l}$ equals that of $W_{k}$ for all $0 \leq k \leq l$, although Theorem \ref{fraenkelho1} only proves this for $0 \leq k <l$. Thus it appears to be hard to find a formula for the $1$-positions of $W_{k,l}$ with $l$ even. But in Theorem $\ref{lodd}$, we are able to provide a formula for the $1$-positions of $W_{k,l}$ with $l$ odd and $k<l$. 

The proof of Theorem \ref{lodd}, the formula for the $1$-positions of $W_{k,l}$ with $l$ odd, will require the four lemmas below. The proofs are omitted as they are elementary and use ideas tangential to the rest of the paper. Lemma's \ref{covering} and \ref{agaps} are used directly as important parts of Theorem \ref{lodd} while Lemma's \ref{fracless} and \ref{fraclesscomplement} are used to prove Lemma \ref{agaps}. 

\begin{lemma}\label{covering}
For $n \ge 1$, let $A_n= \lfloor n \phi \rfloor$ and $B_n= \lfloor n \phi^2 \rfloor$. The following sets partition the set of integers greater than or equal to 2: $\{ A_n | n \in \{B_k\}\} , \{ B_n+1 | n \in \{B_k\}\} , \{A_n+1| n \in \{A_k\}\} , \{B_n+2 | n \in \{A_k\}\} $
\end{lemma}

\begin{lemma}\label{fracless}
For all integers $k \ge 0$, we have $\{ \phi \lfloor k \phi^2 \rfloor \} < 2- \phi$. 
\end{lemma}

\begin{lemma}\label{fraclesscomplement}
For all integers $k>0$, we have $\{ \phi \lfloor k \phi  \rfloor \} \ge  2- \phi$. 
\end{lemma}
 
\begin{lemma}\label{agaps}
We have $\lfloor \phi n \rfloor = \lfloor \phi (n-1) \rfloor + 1$ if and only if $n= \lfloor k \phi^2 \rfloor +1$ for some $k$. 
\end{lemma}

\begin{theorem}\label{lodd}
For $l=2m+1$, $k<l$, the set of $1$-positions $(a,b)$ with $a \leq b$ of $W_{k,l}$ is: $ \{ (2i,2i+1) | 0 \leq i \leq m \} \cup \{ (l+1,l+1) \} \cup \{ (A_n+l, B_n+l+1) | n= \lfloor j \phi^2 \rfloor \text{ for some } j \ge 1 \} \cup \{ (A_n+l+1, B_n+l+2) | n= \lfloor j \phi \rfloor \text{ for some } j \ge 1 \}$. 
\end{theorem}

\begin{proof}
By Theorem \ref{ppositionswkl}, the $0$-positions $(a,b)$ with $a \leq b$ of $W_{k,l}$ are 
\[ S_{0}= \{(i,i) | 0 \leq i <l \} \cup \{(A_n+l, B_n+l)  | n \ge 0 \}. \] 

Let $S_{1}=\{ (2i,2i+1) | 0 \leq i \leq m \} \cup \{ (l+1,l+1) \} \cup \{ (A_n+l, B_n+l+1) | n= \lfloor j \phi^2 \rfloor \text{ for some } j \ge 1 \} \cup \{ (A_n+l+1, B_n+l+2) | n= \lfloor j \phi \rfloor \text{ for some } j \ge 1 \}$. 

It suffices to prove the following:
\begin{itemize}
\item[(a)] $S_{0} \cap S_{1}= \emptyset$
\item[(b)] There is no move from a position in $S_{1}$ to a position in $S_{1}$. 
\item[(c)] From every position not in $S_{0} \cup S_{1}$, there is a move to a position in $S_1$. 
\end{itemize}

We now prove each statement. We first make note of a fact useful in the rest of the proof. 

Suppose $(A_n+l,B_n+l+1)=(A_m+l+1,B_m+l+2)$. For $A_n=A_m+1$, we must have $m=n-1$. Then $B_n=B_{n-1}+1$. This is a contradiction, as gaps in the sequence $\{B_k\}$ always have size at least two. Thus no position can be written both as $(A_n+l,B_n+l+1)$ for some $n$ and $(A_m+l+1,B_m+l+2)$ for some $m$. 

\begin{itemize}
\item[(a)]
Suppose $(A_n+l,B_n+l)=(A_m+l+1,B_m+l+2)$. For $A_n=A_m+1$, we must have $m=n-1$. So $A_n=A_{n-1}+1$. By Lemma \ref{agaps}, we have $n-1=\lfloor k \phi^2 \rfloor$ for some $k$. Since $\{ \lfloor k \phi^2 \rfloor \} $ and $\{ \lfloor j \phi \rfloor \}$ are complementary sequences, we have $n-1 \neq \lfloor j \phi \rfloor$. Since from above we have that $(A_m+l+1,B_m+l+2)$ is not also of the form $(A_j+l, B_j+l+1)$, we have that $(A_m+l+1,B_m+l+2) \notin S_1$. 

We see that we cannot have $(A_n+l,B_n+l)=(A_m+l,B_m+l+1)$ because the first coordinate requires $n=m$ and the second requires $n \neq m$. 

It is clear that $\{ (2i,2i+1) | 0 \leq i \leq m \} \cup \{ (l+1,l+1) \}$ does not intersect $S_0$. 

Thus $S_0 \cap S_1 =\emptyset$. 

\item[(b)]
We first show that no diagonal moves exist between positions in $S_1$. 

A diagonal move cannot be taken to a position in the set $\{ (2i,2i+1) | 0 \leq i \leq m \} \cup \{ (l+1,l+1) \} \}$ because no diagonal moves can be made to a position $(a,b)$ unless $\max(a,b) \ge l$ and because the difference between the two coordinates of positions in $\{ (A_n+l, B_n+l+1) | n= \lfloor j \phi^2 \rfloor \text{ for some } j \ge 1 \} \cup \{ (A_n+l+1, B_n+l+2) | n= \lfloor j \phi \rfloor \text{ for some } j \ge 1 \}$ is at least $2$ because $B_n=A_n+n$. 

We now show that there are no diagonal moves starting at a position of the form $(A_n+l, B_n+l+1)$.

Suppose we subtract $s$ from each coordinate in $(A_n+l, B_n+l+1)$ and reach a position of the form $(A_m+l
+1,B_m+l+2)$, so we have $A_n+l-s=A_m+l+1$ and $B_n+l+1-s=B_m+l+2$. Note that $B_j=A_j+j$ because $\phi^2=\phi+1$. So $A_n+n+l+1-s=A_m+m+l+2$, and subtracting $1$ plus the first equation we have $n=m$.We have assumed $(A_n+l,B_n+l+1) \in S_1$, so we can only have $(A_n+l+1,B_n+l+2) \in S_1$ if $(A_n+l+1,B_n+l+2)=(A_j+l,B_j+l+1)$. But as we showed at the beginning of $(a)$, no position can be written in both forms. 

Suppose we subtract $s$ from each coordinate in $(A_n+l,B_n+l+1)$ and reach a position of the form $(A_j+l,B_j+l+1)$, so $A_n+l-s=A_j+l$ and $B_n+l+1-s=B_j+l+1$. The second equation becomes $A_n+n+l+1-s=A_j+j+l+1$, and then subtracting $1$ plus the first equation, we have $n=j$. But then $s=0$, and thus there is no move. 

Similarly, we can show that there are no diagonal moves starting at a position of the form $(A_n+l+1,B_n+l+2)$.

Next, we show that no nim moves exist between positions in $S_1$. 
Suppose the starting position is $(A_n+l,B_n+l+1)$, and we make a nim move to $(A_n+l,x)$. This position can only be in $S_1$ if $A_n+l=A_m+l+1$, which implies that $m=n-1$. Since $A_n=A_{n-1}+1$, by Lemma \ref{agaps}, we have $n=\lfloor k \phi^2 \rfloor +1$ for some $k$. Since $n-1=\lfloor k \phi^2 \rfloor$ for some $k$, then $n \neq \lfloor j \phi^2 \rfloor$ for any $j$. But then $(A_n+l,B_n+l+1) \notin S_1$. 
Now suppose that starting from $(A_n+l,B_n+l+1)$ we make a nim move to $(x,B_n+l+1)$. Then $B_n+l+1=B_m+l+2$ for some $m$. But this is a contradiction, because no consecutive numbers are in the sequence $\{B_k \}$. 

The proof that there is no nim move starting at a position of the form $(A_n+l+1, B_n+l+2)$ is similar. 

\item[(c)] As there is symmetry in the rule set and thus $g$-positions for $W_{k,l}$ across the line $y=x$, we prove only for positions $(a,b)$ where $a \leq b$. By Lemma \ref{covering}, the set $S_1$ plus the corresponding set of $1$-positions with $a>b$ contains a position at every $x$-coordinate. 
The difference between the $y$-coordinates and $x$-coordinates in each pair $(A_n+l,B_n+l+1)$ and $(A_n+l+1,B_n+l+2)$ is $B_n+1-A_n=n+1$. Every $n \ge 1$ is either of the form $n=\lfloor j \phi \rfloor$ for some $j \ge 1$ or $n=\lfloor j \phi^2 \rfloor$ for some $j \ge 1$. So the set of differences between $y$-coordinates and $x$-coordinates includes all $n \ge 2$. The positions $(2i,2i+1)$ and $(l+1, l+1)$ expand this set of differences to include all $n\ge 0$. Therefore the set $S_1$ contains a position on every diagonal $y=x+j$ for $j \ge 0$. 

Having established the presence of a position in $S_1$ at every $x$-coordinate and on every diagonal, we show that from any position $(a,b)$ with $a \leq b$ such that $(a,b) \notin S_0 \cup S_1$, there is a move to a position in $S_1$. 

Consider the position $(a,b)$ with $a \leq b$. There is some position $(a, b')$ in $S_1$. If $b>b'$, then we make a nim move to $(a,b^{'})$. If $b=b^{'}$, then $(a,b) \in S_1$ and no move is necessary. Suppose $b <b^{'}$. We show that the position $(a^{'},b^{''})$ in $S_1$  on the diagonal $y=x+(b-a)$ satisfies $a^{'}<a$, and thus can be reached from $(a,b)$ by a diagonal move. 
The position in $S_1$ on the diagonal $y=x+(b'-a)$ has $y$-coordinate of $b'$. Since $b<b^{'}$, the position, the diagonal $y=x+(b-a)$ is lower. The difference between the $y$ and $x$ coordinates of a position $(A_n+l,B_n+l+1)$ or $(A_n+l+1,B_n+l+2)$ is $n+1$, and clearly $A_n$ decreases with $n$, so a lower diagonal will have its position in $S_1$ at a smaller $x$-coordinate. We note that the  subset $\{ (l-1,l) \} \cup \{ (l+1,l+1)  \} \cup \{ (A_n+l, B_n+l+1) | n= \lfloor j \phi^2 \rfloor \text{ for some } j \ge 1 \} \cup \{ (A_n+l+1, B_n+l+2) | n= \lfloor j \phi \rfloor \text{ for some } j \ge 1 \}$ of $S_1$ contains a position on each diagonal on or above $y=x$, and that each position $(i,j)$ in this subset satisfies $\max(l-1,l) \ge l$ and $\min(l-1,l) \ge l-1 \ge k$, so no move into the forbidden region is ever required to reach a position in $S_1$ on a given diagonal. 
So the position $(a^{'},b^{''})$ in $S_1$ is reachable by a diagonal move. 
Thus from every position not in $S_0 \cup S_1$, there is a move to a position in $S_1$. 

\end{itemize}

As $S_1$ satisfies properties $(a),(b),$ and $(c)$, $S_1$ is indeed the set of $1$-positions $(a,b)$ with $a \leq b$ of $W_{k,l}$. 

\end{proof}

\section{The Class $T_{k}$}

In this section, we prove Theorem \ref{tk} which gives an invariance property of Sprague-Grundy function for the class $\{T_k\}$. First, we establish some necessary lemmas which bound the location of the $g$-positions in the game $T_K$. These $g$-positions will relate to those of the game $W_1$, as $T_k$ restricts the rule set of $W_1$. 

\begin{table}[h!]
\centering
\caption{Nim Values of $W_1$}
\label{my-label}
\begin{tabular}{l|llllllllllllllllllllllll}
3 & 3 & 2 & 0 & 4 & 1 & 8 & 9 & 10 & 5 & 7 & 6 & 12 & 15 & 11 & 16 & 17 & 13 & 20 & 14 & 21 & 18 & 19 &24 & 25 \\
2 &2 &3  &1  &0  & 6 & 7 & 5 & 4 & 10 & 11 & 9 & 8 & 14 & 15 &13 & 12 & 18 & 19 & 17 & 16 &22  &23  & 21 & 20  \\
1 & 1& 0 & 3 & 2 & 5 & 4 & 7 & 6 & 9 & 8 & 11 & 10 & 13 & 12 & 15 & 14 & 17 & 16 & 19 & 18 & 21 & 20 & 23 & 22  \\
0 & 0 & 1 & 2 & 3 & 4 & 5 & 6 & 7 & 8 & 9 & 10 & 11 & 12 & 13 & 14 & 15 &16  & 17 & 18 & 19 & 20 & 21 & 22 & 23 \\ \hline
 & 0 & 1 & 2 & 3 & 4 & 5 & 6 & 7 & 8 & 9 & 10 & 11 & 12 & 13 & 14 & 15 &16  & 17 & 18 & 19 & 20 & 21 & 22 & 23
\end{tabular}
\end{table}

\begin{lemma}\label{yequalsthree}
Let $(x,y)$ be the position with nim-value $0$ for a given $y$. We consider for which $g$ the $g$-position with the given $y$-coordinate occurs at $(x+g,y)$. 
For $y=0$, this occurs for all $g$. For $y=1$, this occurs exactly when $g$ is even. 
For $y=2$, this occurs exactly when $g \equiv 0$ mod $4$. 
For $y=3$, this does not occur for any $g \equiv 0$ mod $4$. 
\end{lemma}

\begin{proof}
For $y=0$ and $y=1$, the nim-value of the position $(x,y)$ is $x \oplus y$. This is because no diagonal moves may occur, so the nim-values will equal those of $2$-pile Nim. So the $0$-position with $y=0$ is at $(0,0)$ and the $g$-position with $y=0$ is at $(0,g)$ for all $g$. Also, the $0$-position with $y=1$ is at $(0,1)$, and the $g$-position with $y=1$ is at $(0,g+1)$ for $g$ even and $(0,g-1)$ for $g$ odd. 

For $y=2$, we can compute that the $0$-position is at $(3,2)$.The pattern for the $g$-positions depends on the residue class mod $4$ of the $x-$coordinate. If $x \equiv 0 \mod 4$ or $x \equiv 1 \mod 4$, then the nim-value of $(x,2)$ is $x \oplus 2$. If $x \equiv 2 \mod 4$, then the nim-value of $(x,2)$ is $(x \oplus y) + 1$, and if $x \equiv 3 \mod 4$, the the nim-value of $(x,2)$ is $(x \oplus y)-1$. We prove inductively, showing that this holds in the interval $[4k,4k+3]$ if it holds for $[0,4k-1]$. The proof proceeds by applying the definition of a nim-value, which for a nonterminal position $(a,b)$ is the minimum excluded natural number of the set of positions reachable in one move from $(a,b)$. The base case can be computed. 
Suppose the formula holds for $x \in [0,4k-1]$. The nim values in $[0,4k-1]$ can be reached because they occur in the row $y=2$ for smaller $x$ by the induction hypothesis. At $x=4k$, the nim-values $4k$ and $4k+1$ occur lower in that column by the formula for $g$-positions at $y=0$ and $y=1$. By the inductive hypothesis and the formulas for $g$-positions for $y=0$ and $y=1$, there are no positions with nim-value $4k+2$'s reachable in one move. Thus the minimum excluded integer is $4k+2$. The proof of the rest of the formula uses the same idea and thus is omitted. So we have the $0$-position located at $(3,2)$ and a $g$-position located at $(g+3,2)$ exactly when $g \equiv 0$ mod $4$. 

Similarly, we can prove that for $y=3$ and $x \ge 20$, the following pattern holds for $g$-positions, which depends on the residue class mod $4$ of the $x$-coordinate. If $x \equiv 0$ or $x \equiv 1$ mod $4$, then the nim-value of $(x,3)$ is $x-2$. If $x \equiv 2$ or $x \equiv 3$ mod $4$, the nim-value of $(x,3)$ is $x+2$. We compute the nim-values for $x<20$ separately, using the recursive definition of a nim-value. We find that the $0$-position is at $(2,3)$ and no $g$-position with $g \equiv 0$ mod $4$ is at $(g+2,3)$. 
\end{proof}

\begin{lemma}\label{pushingaway}
Suppose $(x,y)$ is a $g$-position with $x<y, y=\lfloor \phi n \rfloor +1$ for some $n$. Then $x \leq \lfloor \phi^2 n  \rfloor+1$. 
\end{lemma}

\begin{proof}
The $0$-position with $y=\lfloor \phi n \rfloor +1$ is at $(\lfloor \phi^2 n  \rfloor+1, \lfloor \phi n \rfloor +1)$. The lemma states that any $g$-position with the same $y$ and larger $x$ will be no further than $g$ to the right of the $0$. We prove by considering the recursive definition of a nim-value, noting that the location of a $g$-position in a given row will depend only on the location of the $l$-positions for $0 \leq l<g$ in that row and the $g$-positions in lower rows. 

Consider placing $0$-positions row by row. The $0$ in a given row will be at the smallest $x$-coordinate such that there is not already a $0$ with that $x$-coordinate or on that diagonal for some smaller $y$. Next we similarly place $g$-positions row by row, in order of increasing $g$, now avoiding not only rows and diagonals already containing $g$-positions but also positions already filled with a $0,1,\dots,$ or $g-1$. The first $x$-coordinate where it is possible to place a $g$ will be largest if the rows and diagonals already containing $g$ are shifted $g$ to the right of those containing $0$-positions. Otherwise, in some $x$-coordinate before that $g$ to the right of the $0$-position it is possible to place a $g$. For the rows and diagonals containing $g$ to be shifted $g$ to the right, the $g$ in each row must be $g$ to the right of the $0$ that row. By Lemma \ref{yequalsthree} this can happen at most for the consecutive rows $y=0, y=1$ and $y=2$, and in these rows, we know the location of $g$-positions and none is more than $g$ to the right of $0$. So there is no row in which a $g$ is more than $g$ to the right of $0$.  
\end{proof}

\begin{lemma}\label{glocation}
All $g$-positions above the line $y=x$ are under the line $y=\phi x+g$. Excluding the $g$-position with $b=0$, all $g$-positions above the line $y=x+g$ are to the right of $y=(g+1)x$. 
\end{lemma}

\begin{proof}
First, we show that all $0$-positions are under $y=\phi x$. From Theorem 7 in \cite{FH}, we have that the $0$-positions of $W_1$ and $T_k$ for all $k \ge 0$ form the set $\{ (0,0) \} \cup \{ (A_n+1,B_n+1) | n \ge 0 \}$. 
Let $x=\lfloor \phi n \rfloor +1$. We have $y=\lfloor \phi^2 n \rfloor +1 < \phi \lfloor \phi n \rfloor +1 =\phi (x-1) +1= \phi x +(1-x) < \phi x$. Thus, all $0$-positions are under the line $y=\phi x$, so by Lemma \ref{pushingaway}, all $g$-positions above the line $y=x$ are under the line $y=\phi x+g$. 

Next, we show that excluding the $g$-position with $b=0$, all $g$-positions above the line $y=x+g$ are to the right of $y=(g+1)x$. The intersection of $y=x+g$ and $y=(g+1)x$ is at $x=1$. The intersection of $\phi x+g$ and $(g+1)x$ is at $x=\frac{g}{g+1-\phi}$. For $g \ge 2$, $\frac{g}{g+1-\phi}<2$. So we only must show that at $x=1$ the $g$ is less than or equal to $g+1$. No diagonal moves are allowed up to this point, so the nim-values will equal those of $2$-pile Nim, which can be computed via the nim sum. So the $g$ is no higher than $g+1$ because $y \oplus 1 \leq y+1$. 
This proves the lemma. 
\end{proof}

\begin{figure}\label{tkplot}
\caption{Regions Used in Proof of Theorem \ref{tk} for g=2}
\includegraphics[scale=0.8]{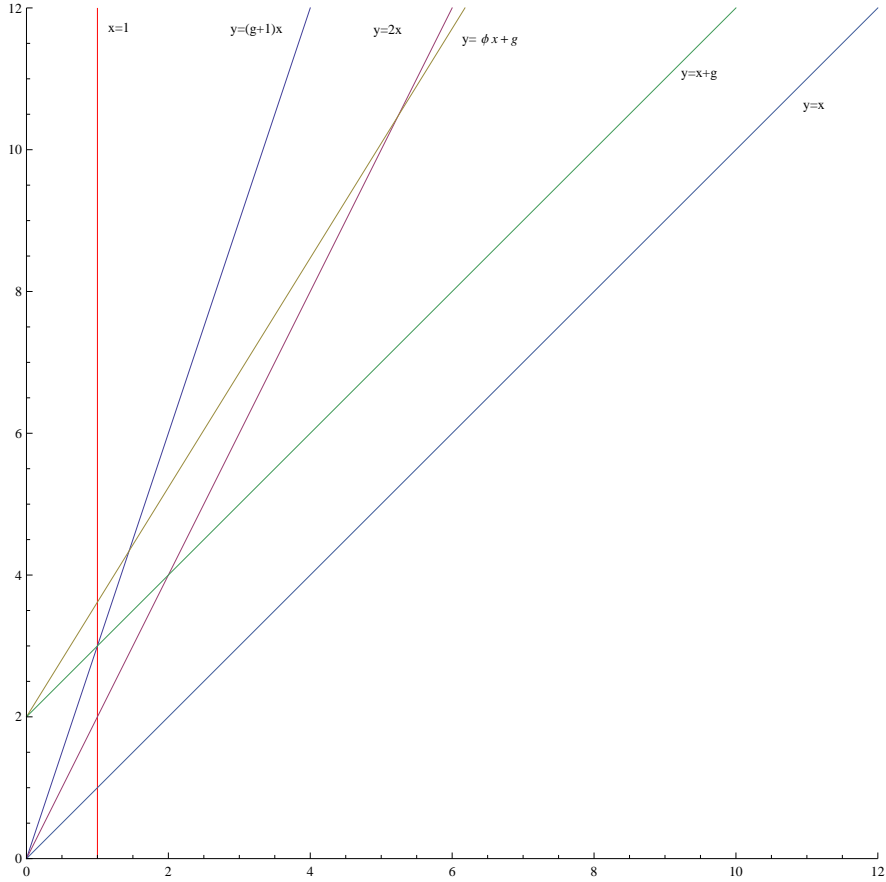}
\end{figure}

\begin{theorem}\label{tk}
(Conjectures 2 and 3 of Fraenkel and Ho). Let $k$ be a nonnegative integer. For every integer $g$ in the range $0 \leq g \leq k$, the two games $T_k$ and $W_1=T_{\infty}$ have the same set of positions with nim-value $g$. In particular, if $k$ and $l$ are nonnegative integers, for every integer $g$ in the range $0 \leq g \leq \min(k,l)$, the two games $T_k$ and $T_l$ have the same set of positions with nim-value $g$. 
\end{theorem}

\begin{proof}
We show that the $g$-positions for $0 \leq g \leq k$ are the same in the game $T_k$ as in $W_1$. The games $T_k$ and $W_1$ all allow the same Nim moves, so the location of the $g$-positions could differ only if there were positions containing $0,1, \dots, g-1$ positions that were reachable via diagonal move in one game but not the other. The location of the $g$-positions is bounded in a region described in Lemma \ref{glocation}. We show that all positions in this region, which depends on $g$, can be reached via the diagonal move in all games $T_k$ where $0 \leq g \leq k$. Thus the $g$-positions will be the same for these games. For $a \leq b$ we consider the expression $ \left| \lfloor \frac{b-s}{a-s} \rfloor - \lfloor \frac{b}{a} \rfloor \right|$, and for $a>b$, we consider the expression $ \left| \lfloor \frac{a-s}{b-s} \rfloor - \lfloor \frac{a}{b} \rfloor \right|$, so the location of $g-$positions is symmetrical across the line $y=x$. Thus we discuss only the case where $a \leq b$. 

We condition on a position's location with respect to the line $y=2x$, showing that in each case, that position has unrestricted diagonal access to all relevant $g$-positions. (See Figure \ref{tkplot} for a visualization of the difference regions.) A $g$-position with $b=0$ is not accessible via diagonal from any position in any game $T_k$ or $W_1$. Consider a position $(a,a+j)$, and the corresponding $ \left| \lfloor \frac{a+j-s}{a-s} \rfloor - \lfloor \frac{a+j}{a} \rfloor \right|=\left| \lfloor 1+\frac{j}{a-s} \rfloor - \lfloor 1+\frac{j}{a} \rfloor  \right|$. Since $ 1 \leq \lfloor 1+\frac{j}{a-s} \rfloor \leq j+1$ and $1 \leq \lfloor 1+\frac{j}{a} \rfloor  \leq j+1$, we have  $\left| \lfloor \frac{a+j-s}{a-s} \rfloor - \lfloor \frac{a+j}{a} \rfloor \right| \leq j$, so movement along the $y=x, y=x+1, \dots, y=x+g$ diagonals for $0 \leq g \leq k$ is unrestricted in the game $T_k$, so a $g$-position on or below the line $y=x+g$ is accessible to any position on its diagonal in those games. By Lemma $\ref{glocation}$, all that remains is to show that for other diagonals, diagonal movement as far as the line $y=(g+1)x$ is unrestricted.  

First, we show that all positions to the right of $y=2x$ can reach to $y=(g+1)x$. Let $(a,a+j)$ be a position on the diagonal $y=x+j$ to the right of $y=2x$. Then $a+j<2a$, so $j<a$. We have $ \left| \lfloor \frac{a+j-s}{a-s} \rfloor - \lfloor \frac{a+j}{a} \rfloor \right|= \left|  \lfloor \frac{j}{a-s} \rfloor - \lfloor \frac{j}{a} \rfloor  \right| = \lfloor \frac{j}{a-s} \rfloor.$ So, for $a-s>0$, if $a-s \ge \frac{j}{k}$, then $\lfloor \frac{j}{a-s} \rfloor \leq \frac{j}{a-s} \leq k$, and a diagonal move from $(a,a+j)$ to $(a-s, a-s+j)$ is legal in $T_k$. So for all $(a,a+j)$ to the right of $y=2x$, it is legal to move as far on the diagonal $y=x+j$ as $x=\frac{j}{k}$. That is, it is legal to move as far as the point $(\frac{j}{k}, \frac{j}{k}+j)=(\frac{j}{k}, \frac{(k+1)j}{k})$, or the line $y=(k+1)x$. So for $0 \leq g \leq k$, a move is legal as far as to the line $y=(g+1)x$. 

Second, we show that all positions to the left of or on $y=2x$ (and the right of $y=(g+1)x$) can reach to $y=(g+1)x$. 
Let $(a,a+j)$ be a position on the diagonal $y=x+j$ on or to the left of $y=2x$. Then $a+j \ge 2a$, so $j \ge a$. We have $
 \left| \lfloor \frac{a+j-s}{a-s} \rfloor - \lfloor \frac{a+j}{a} \rfloor \right|= \left|  \lfloor \frac{j}{a-s} \rfloor - \lfloor \frac{j}{a} \rfloor  \right|= \lfloor \frac{j}{a-s} \rfloor - \lfloor \frac{j}{a} \rfloor$. We seek to satisfy $\lfloor \frac{j}{a-s} \rfloor - \lfloor \frac{j}{a} \rfloor \leq k$. Since $ \lfloor \frac{j}{a} \rfloor \ge 1$, a stronger condition than the above is $\lfloor \frac{j}{a-s} \rfloor \leq k+1$, and stronger than this is $ \frac{j}{a-s} \leq k+1$. So, for $a-s>0$, if $a-s \ge \frac{j}{k+1}$, then a diagonal move from $(a,a+j)$ to $(a-s, a-s+j)$ is legal in $T_k$. Thus for all $(a,a+j)$ to the left of or on $y=2x$, it is legal to move as far on the diagonal $y=x+j$ as $x=\frac{j}{k+1}$, that is, to the point $(\frac{j}{k+1}, \frac{(k+2)j}{j})$, or the line $y=(k+2)x$. 
This proves the theorem. 
\end{proof}

\section{Acknowledgments}

This research was conducted as part of the 2015 Duluth REU program and was supported by NSF grant 1358695,  NSA grant H98230-13-1-0273, and University of Minnesota Duluth.  I would like to thank the participants, advisers Levent Alpoge and Ben Gunby, program director Joe Gallian, and visitors Tim Chow, Albert Gu, Adam Hesterberg, and Alex Lombardi of the Duluth REU for many helpful discussions.

\ \\

\end{document}